\documentclass[draft]{amsart}
\usepackage{amssymb}
\usepackage{url}
\usepackage{amscd}
\usepackage{graphicx}
\input epsf
 
 \theoremstyle{plain}
 \newtheorem{theorem}{Theorem}
 
 \newtheorem{corollary}[theorem]{Corollary}
 \newtheorem{lemma}[theorem]{Lemma}
 
 \theoremstyle{remark}
 \newtheorem{remark}{Remark}
 
 \numberwithin {equation}{section}
 
\begin{document}
%\date{Received ?; received in final form ?}
\title[Extension Theorems]{Extension Theorems for spheres in the finite field setting}
\author{Alex Iosevich and Doowon Koh}

\address{Mathematics Department\\
202 Mathematical Sciences Bldg\\
University of Missouri\\
Columbia, MO 65211 USA}
\email{iosevich@math.missouri.edu}
\address{Mathematics Department\\
202 Mathematical Sciences Bldg\\
University of Missouri\\
Columbia, MO 65211 USA}

\email{koh@math.missouri.edu}

\thanks{This work was partially supported by the NSF Grant DMS04-56306. }

%\subjclass{}

\begin{abstract}
In this paper we study the boundedness of extension operators associated with spheres in vector spaces over finite fields.
In even dimensions, we estimate the number of  incidences  between spheres and points in the translated set from a subset 
of spheres. As a result, we improve the Tomas-Stein exponents, previous results by the authors in \cite{IK07}. 
The analytic approach and the explicit formula for  Fourier transform of the characteristic function 
on spheres play an important role to get  good bounds for exponential sums. 
\end{abstract}

\maketitle

%%%%%%%%%%%%%%%%%%%%%%%%%%%%%%%%%%%%%%%%%%%%%%%%%%%%%%%%%%%%%%%%%%%%%%%%%
% Macros
%%%%%%%%%%%%%%%%%%%%%%%%%%%%%%%%%%%%%%%%%%%%%%%%%%%%%%%%%%%%%%%%%%%%%%%%%
 
\newcommand\sfrac[2]{{#1/#2}}
\newcommand\cont{\operatorname{cont}}
\newcommand\diff{\operatorname{diff}}

%%%%%%%%%%%%%%%%%%%%%%%%%%%%%%%%%%%%%%%%%%%%%%%%%%%%%%%%%%%%%%%%%%%%%%%%%
 
\section{Introduction}
In the Euclidean setting, extension theorems (or restriction theorems) address the problem of finding  
the exponents $p$ and $r$ such that the estimate
\[\| \widehat{fd\sigma} \|_{L^r({\mathbb R}^d)} \leq C_{p,r} \|f\|_{L^p(S, d\sigma)}
\quad \mbox{ for all } f \in L^p(S, d\sigma) \]
holds, where $S$ is a hypersurface of  ${\mathbb R}^d$ and $d\sigma$ is a surface measure on $S$. 
Since this problem was introduced by E. M. Stein in 1967, it has been extensively studied. 
See, for example, \cite{Fe70},\cite{ Zy74},\cite{Str77},\cite{ Ste93},\cite{ Gr03}. A comprehensive survey of this problem is given in \cite{Ta03}.
On the other hand, Mockenhaupt and Tao (\cite{MT04}) recently studied extension theorems 
in the finite field setting for various algebraic varieties $S$. Their work was mostly restricted 
to cones and paraboloids in vector spaces over finite fields and they point out that the Fourier transform of surface measure on spheres is not as easy to compute in the finite field case. 
The authors of this paper obtained in \cite{IK07} the sharp decay of Fourier transform of  non-degenerate quadratic surfaces and used it to obtain the Tomas-Stein exponents for the corresponding extension problems. Here and throughout, the Tomas-Stein exponents are pairs $(p,r)$ such that the following two inequalities hold:
\[ r\geq \frac{2d+2}{d-1} \quad \mbox{and} \quad r\geq \frac{p(d+1)}{(p-1)(d-1)} ,\]
(see Figure 1). 
However this result seems to be far from the best one we can expect. The purpose of this paper is to significantly improve the Tomas-Stein exponents 
in the specific case when the non-degenerate quadratic surfaces under consideration are spheres in even dimensional vector spaces over finite fields. We begin by recalling some notation and Fourier analytic machinery in the finite field setting. 
We denote by $ {\mathbb F}_q^d$  a d-dimensional vector space over the finite field, ${\mathbb F}_q $, with $q$ elements. 
We assume that the characteristic of ${\mathbb F}_q $ is greater than two. In other words, $q$ is a power of an odd prime. 
For each $j \in {\mathbb F}_q^* = {\mathbb F}_q \setminus \{0\},$ we define the sphere $S_j$ in ${\mathbb F}_q^d$ by the relation
\begin{equation}\label{defofsphere} S_j = \{ x \in {\mathbb F}_q^d : \| x\|_2 =j \}, \end{equation}
where the notation  $\| .\|_2$ is defined by the relation $ \|x\|_2 = x_1^2 + \ldots + x_d^2$. 
In the sense of analysis, the $\| .\|_2$ is not a norm since it is not real-valued,
but it can be a norm in the sense of algebra (see, for example,  \cite{Li}, p.356 ). 

We fix $\chi : {\mathbb F}_q \rightarrow  \mathbb C$ to be a non-trivial additive character of ${\mathbb F}_q$.
If $q$ is a prime, we may choose $\chi(s) = e^{2\pi i s/ q}$ and the exact choice of the non-trivial character is independent 
of the results in this paper. Given a complex-valued function $f$ on ${\mathbb F}_q^d$, $d \geq 1,$ we define the Fourier transform of $f$ by the formula
\[ \widehat{f}(m) = q^{-d} \sum_{x \in {\mathbb F}_q^d} \chi(-x\cdot m) f(x) \] where 
$x\cdot m$ is the usual dot product of $x$ and $m$. Similarly, we define the Fourier transform of the measure $f d\sigma$ by the relation
\[ \widehat{fd\sigma}(m) = \frac{1}{|S|} \sum_{x \in S} \chi(-x\cdot m) f(x) \]
where $|S|$ denotes the number of elements in  an algebraic variety $S$ in ${\mathbb F}_q^d$, and  $d\sigma$ denotes normalized surface measure on $S$.
Using the orthogonality relations for non-trivial characters, we obtain the Fourier inversion theorem, that is,
\[ f(x)=\sum_{m\in {\mathbb F}_q^d} \chi(x\cdot m) \widehat{f}(m).\]
Given complex-valued functions $f,g$ on ${\mathbb F}_q^d$, the Plancherel theorem is given by 
\[ \sum_{m \in {\mathbb F}_q^d} \widehat{f}(m) \overline{\widehat{g}}(m) = q^{-d}\sum_{x \in {\mathbb F}_q^d} f(x) \overline{g}(x).\]  
Note that the Plancherel theorem says in this context that
\[ \sum_{m \in {\mathbb F}_q^d} |\widehat{f}(m)|^2 = q^{-d}\sum_{x \in {\mathbb F}_q^d} |f(x)|^2.\]  
Endow the measure on the ``space" variables, dx, with the normalized counting measure given by dividing  
the counting measure by $ q^d$, and the measure on the ``phase" variables, dm, with the usual counting measure. 
Then  we obtain the following definitions : for $1\leq p, r < \infty,$
\[\|f\|_{L^p \left({\mathbb F}_q^d,dx \right)}^p=q^{-d}\sum_{x\in {\mathbb F}_q^d}|f(x)|^p,\]
\[\|\widehat{f}\|_{L^r \left({\mathbb F}_q^d,dm \right)}^r = \sum_{m\in
{\mathbb F}_q^d}|\widehat{f}(m)|^r\]
 and
\[ \|f\|_{L^p \left(S,d\sigma \right)}^p=\frac{1}{|S|}\sum_{x\in S}|f(x)|^p.\] 
Similarly, denote by  $\|f\|_{L^\infty}$ the maximum value of $|f|$.

\subsection{Statement of the main result}
Given an algebraic variety $S$ in ${\mathbb F}_q^d$ and the surface measure on $S$ denoted by $d\sigma$, 
we define $R^*(p \rightarrow r)$ as the best constant such that the extension estimate 
\begin{equation} \label{defExtension}\|\widehat{fd\sigma}\|_{L^r ({\mathbb F}_q^d,dm)}\leq R^*(p\rightarrow r)
\|f\|_{L^p (S,d\sigma)}\end{equation}
holds  for all functions $f$ on $S$. The constant $R^*(p \rightarrow r)$ may depend on the underlying field ${\mathbb F}_q.$
However, the extension theorem asks us to determine the exponents $p$ and $r$ such that $R^*(p \rightarrow r) \lesssim 1$ where 
the constant in the inequality is independent of the size, ``$q$", of the underlying field ${\mathbb F}_q.$ 
Recall that  $X \lesssim Y$ denotes the estimate
$X \leq CY$ where the constant $C$ is independent of $q$. $X\approx Y$ means that
both $X \lesssim Y$ and $Y \lesssim X$ hold. We also recall that $X \lessapprox Y$ is used 
if $X \leq C_\varepsilon q^\varepsilon Y$ for all $\varepsilon >0$ , where 
$ C_\varepsilon$ is independent of $q$.

In two dimensions, the extension theorems for the parabolas and the circles in ${\mathbb F}_q^d$ 
were completely solved by the authors in \cite{MT04} and the authors in \cite{IK07} respectively. 
In higher dimensions, $d\geq 3$, Mockenhaupt  and Tao (\cite{MT04}) also obtained the Tomas-Stein exponents by showing that 
$R^*(2 \rightarrow r) \lesssim 1$ whenever $r \geq (2d+2)/(d-1)$ if $S=\{(x, x\cdot x) \in {\mathbb F}_q^d : x \in {\mathbb F}_q^{d-1}\}$, 
an analog of the Euclidean paraboloid. In particular, when $d=3,$ and $-1$ is not a square in ${\mathbb F}_q$, 
they improved the $R^*(2 \rightarrow 4) $ result from the Tomas-Stein exponents by showing that 
$R^*(8/5 \rightarrow 4) \lessapprox 1$. The main idea for the improvement was the incidence theorem between lines and points 
in ${\mathbb F}_q^2$. However, it could be difficult to improve the Tomas-Stein exponents in higher dimensions, in part, 
because it is not easy to get good incidence theorems in higher dimensions. 
In this paper, we discuss this issue by studying the extension theorem for spheres in vector spaces over finite fields. 
In higher even dimensions, $d\geq 4$, we give the ``$p$" index improvement of the Tomas-Stein exponents on which $R^*(p \rightarrow r) \lesssim 1.$ 
Moreover, we show that in higher odd dimensions, $d\geq 3,$ it is impossible to improve the extension theorems
by the Tomas-Stein exponents if $r\geq (2d+2)/(d-1)$ and $-1$ is a square number in the underlying field ${\mathbb F}_q.$ 

Our main result is the following. 
\begin{theorem} \label{mainThm} Let $S_j$ be a sphere in ${\mathbb F}_q^d$ defined as in (\ref{defofsphere}). 
Suppose $d\geq 4$ is even and $p\geq  \frac{12d-8}{9d-12}.$ Then we have
\[R^*(p \rightarrow 4) \lessapprox 1\] \end{theorem}
\begin{remark} The authors in \cite{IK07} obtained the Tomas-Stein exponents by studying the extension theorems 
for non-degenerate quadratic surfaces in ${\mathbb F}_q^d$. By means of a nonsingular linear substitution, 
we may identify the non-degenerate quadratic surface in ${\mathbb F}_q^d$ with the set of the form
$S_a = \{x\in {\mathbb F}_q^d : a_1x_1^2 + \ldots + a_d x_d^2 =j\}$ where all $a_k  , k=1, \ldots d,$ are not zero in ${\mathbb F}_q$. 
Using the same argument as in the proof of  Theorem \ref{mainThm}, we can also obtain the same result as  Theorem \ref{mainThm}  
in the case when the sphere is replaced by the set $S_a$ above. Thus Theorem \ref{mainThm} can be extended to the extension theorems 
for non-degenerate quadratic surfaces in ${\mathbb F}_q^d$.\end{remark} 
Let us visualize the results of Theorem \ref{mainThm} (see Figure \ref{Figure1} below). 
From H\"older's inequality and the nesting properties of $L^p$-norms, we see that
\begin{equation}\label{smallgood}R^*(p_1\rightarrow r) \le R^*(p_2\rightarrow r) \quad \mbox{for}\quad \ p_1\ge p_2 \end{equation}
and
\[R^*(p\rightarrow r_1)\le R^*(p\rightarrow r_2) \quad \mbox{for}\quad r_1\ge r_2. \] 
In particular, (\ref{smallgood}) implies that our result $R^*( \frac{12d-8}{9d-12}\rightarrow 4) \lessapprox 1$ in Theorem \ref{mainThm} 
is much better than the result, $R^*( \frac{4d-4}{3d-5}\rightarrow 4)\lesssim 1$,  obtained by the Tomas-Stein exponents,  
because the number $ \frac{12d-8}{9d-12}$ is less than $\frac{4d-4}{3d-5}.$
In addition, note that $ R^*(1 \rightarrow \infty) \lesssim 1$ which is the trivial estimate.
Using these facts and interpolation theorems, we conclude that Theorem \ref{mainThm} implies that the bound 
$R^*(p \rightarrow r) \lesssim 1$ holds whenever the pair $(1/p, 1/r)$ is contained in the pentagon  in  Figure \ref{Figure1}.

\begin{figure}[h]
\centering\leavevmode\epsfysize=4.5cm \epsfbox{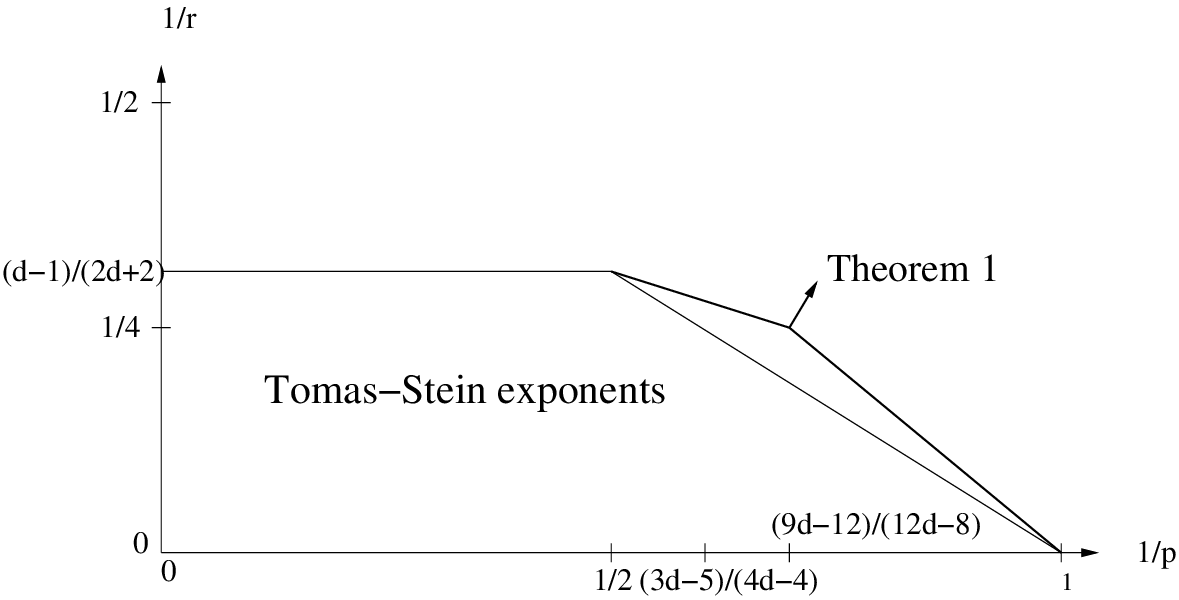}
\caption{\label{Figure1} Tomas-Stein exponents and the improved exponents by Theorem \ref{mainThm}. }
\end{figure}

\subsection{Outline of this paper}
This paper will be constructed as follows.
In Section 2, we shall study the necessary conditions for $R^*(p \rightarrow r) \lesssim 1$ related to spheres. 
As a result, we shall see that, in odd dimensions, $d\geq 3$, one can not expect the ``$p$" index improvement of Tomas-Stein exponents 
on which $R^*(p \rightarrow r) \lesssim 1$ for each $r\geq (2d+2)/(d-1)$ without any restrictions.
In Section 3, we introduce some theorems which can be obtained by applying results from estimates of classical exponential sums. 
The theorems shall play an important role to deal with the key estimates for the proof of Theorem \ref{mainThm}. 
In Section 4, We sketch the proof of Theorem \ref{mainThm}, our main theorem.  
In Section 5, we shall establish the estimates related to dot products determined by a subset of sphere $S_j$, 
which can be used to obtain the key estimate for the proof 
of Theorem \ref{mainThm}. In the last section, we complete the proof of Theorem \ref{mainThm} by giving the proof of Theorem \ref{Long}.

\section{Necessary conditions for $R^*(p\rightarrow r) \lesssim 1$}
In this section, we shall investigate the necessary conditions for the boundedness of the extension operators for spheres $S_j$ defined as in (\ref{defofsphere}). 
Let $S$ be an algebraic variety in $ \mathbb{F}_q^d.$ We assume that $|S| \approx q^\alpha$ for some $ 0<\alpha < d.$
Mockenhaupt and Tao (\cite{MT04}) showed that the necessary conditions for the boundedness of $R^*(p\rightarrow r) $ are given by 
\begin{equation} \label{Necessary1}
r\geq \frac{2d}{\alpha} \quad \mbox{and} \quad r\geq \frac{dp}{\alpha (p-1)}. \end{equation}
However, if the algebraic variety $S$ contains an affine subspace $H \subset \mathbb{F}_q^d$ of dimension $k \,( |H|= q^k)$, then 
we can improve the necessary conditions in (\ref{Necessary1}) by testing 
(\ref{defExtension}) with the characteristic function $H(x)$ on the affine subspace $H$. In fact, if $S$ contains an affine subspace $H\subset \mathbb{F}_q^d$
with $|H|=q^k$, the necessary conditions for $ R^*(p\rightarrow r) \lesssim 1$ are given by 
\begin{equation}\label{Necessary2}  
r\geq \frac{2d}{\alpha} \quad \mbox{and} \quad r\geq\frac{p(d-k)}{(p-1)(\alpha -k)}.\end{equation}
See (\cite{MT04}, pp. 41-42) for the detailed proofs of above necessary conditions.
In the case when the algebraic variety $S$ is the sphere $S_j$, we shall find the necessary conditions for boundedness of $R^*(p\rightarrow r).$ we need the following theorem.

\begin{theorem}\label{subspaceH} Let $H \subset \mathbb{F}_q^d$ be an affine subspace of dimension $k$. Then we have
\[ |H\cap S_j| \lesssim q^{k-1} + q^{\frac{d-1}{2}}.\] \end{theorem}

\begin{proof} Using the Plancherel theorem, we see that
\begin{align*} |H \cap S_j| =& \sum_{x\in \mathbb{F}_q^d} H(x) S_j(x) = q^d \sum_{m\in \mathbb{F}_q^d} \widehat{H}(m) \widehat{S_j}(m)\\
                            =& q^d \widehat{H}(0,\ldots,0) \widehat{S_j}(0,\ldots,0) + q^d \sum_{m\ne (0,\dots,0)} \widehat{H}(m) \widehat{S_j}(m)\\
                             =& I +II.\end{align*}
Since $ |H|= q^k$ and $ |S_j| \approx q^{d-1}$, we obtain that
\begin{equation}\label{First} I= q^d \frac{|H|}{q^d} \frac{|S_j|}{q^d} \approx q^{k-1}. \end{equation}
On the other hand, we observe that
\begin{align*} |II|\leq &q^d \max_{\theta\neq (0,\ldots,0)} |\widehat{S_j}(\theta)| \sum_{m\in \mathbb{F}_q^d} |q^{-d} \sum_{x\in H} \chi(-x\cdot m) |\\
                    \lesssim & q^d q^{-\frac{d+1}{2}} q^{-d} q^{d-k}|H| = q^{\frac{d-1}{2}},\end{align*}
where we used the facts that $H$ is an affine subspace of dimension $k$ and $|\widehat{S_j}(m)| \lesssim q^{-\frac{d+1}{2}} $ if $m \neq (0,\ldots, 0)$ 
(see Remark \ref{boundforsphere}). Combining this with (\ref{First}), the proof immediately follows.
 \end{proof}
From Theorem \ref{subspaceH}, we obtain the following corollary.
\begin{corollary}\label{Second} Let $H \subset \mathbb{F}_q^d$ be an affine subspace of dimension $k$. Moreover, we assume that 
$H \subset S_j$. Then we have 
\[ |H| \lesssim q^{\frac{d-1}{2}}.\]
In addition, if $d \geq 2$, the dimension of $ \mathbb{F}_q^d$, is even, then we have
\[ |H| \lesssim q^{\frac{d-2}{2}}.\]\end{corollary}

\begin{proof} The first part of Corollary \ref{Second} clearly follows from Theorem \ref{subspaceH} and the second part of Corollary \ref{Second}
 follows from the fact that the dimensions of the affine subspaces are non-negative integers.
\end{proof}

If -1 is a square in $\mathbb{F}_q$ and $d$ is odd, there exists a $\frac{d-1}{2}$-dimensional affine subspace $H$ contained in the sphere $S_j$ in $\mathbb{F}_q^d$
(see, e.g., Example 4.4 in \cite{IR06}). Thus if $d$ is odd, then the necessary conditions in (\ref{Necessary2}) take the form
\[ r\geq \frac{2d}{d-1} \quad \mbox{and} \quad r\geq \frac{p(d+1)}{(p-1)(d-1)},\]
because $|S_j| \approx q^{d-1}$ and $|H|= q^{\frac{d-1}{2}}.$ Recall that the Tomas-Stein exponents with $R^*(p\rightarrow r)\lesssim 1$ take the form
\[ r\geq \frac{2d+2}{d-1} \quad \mbox{and} \quad r\geq \frac{p(d+1)}{(p-1)(d-1)} ,\]
which was proved by the authors in \cite{IK07}. Thus if $d \geq 3$ is odd and $r\geq \frac{2d+2}{d-1}$ then the Tomas-Stein exponents give sharp 
``$p$" values such that $R^*(p\rightarrow r)\lesssim 1.$ For example, if $r=4$, we can not improve $R^*(\frac{4d-4}{3d-5} \rightarrow 4) \lesssim 1,$
the result given by Tomas-Stein exponents (see Figure \ref{Figure2} below).

\begin{figure}[h]
\centering\leavevmode\epsfysize=4.5cm \epsfbox{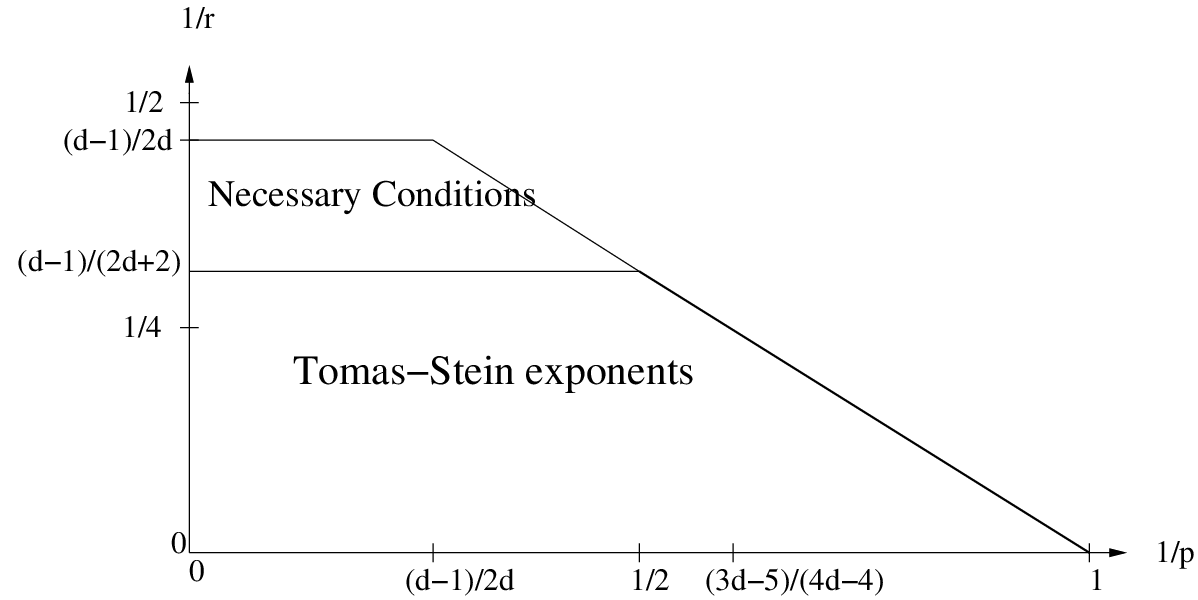}
\caption{\label{Figure2}Necessary conditions for boundedness of $R^*(p\rightarrow r)$ in odd dimensions, $d\geq3.$ }
\end{figure}

However if $d \geq 2 $ is even, we can improve the Tomas-Stein exponents, because the sphere $S_j$ contains at most a $\frac{d-2}{2}$-dimensional affine subspace $H$, 
which is a result from Corollary \ref{Second}. From this and (\ref{Necessary2}), we may conjecture, in even dimensions $d\geq 2,$ that $ R^*(p\rightarrow r) \lesssim 1$
if 
\[ r\geq \frac{2d}{d-1}\quad \mbox{and} \quad r\geq \frac{p(d+2)}{(p-1)d}.\]
Theorem \ref{mainThm} partially supports above conjecture (see Figure \ref{Figure3}).

\begin{figure}[h]
\centering\leavevmode\epsfysize=4.5cm \epsfbox{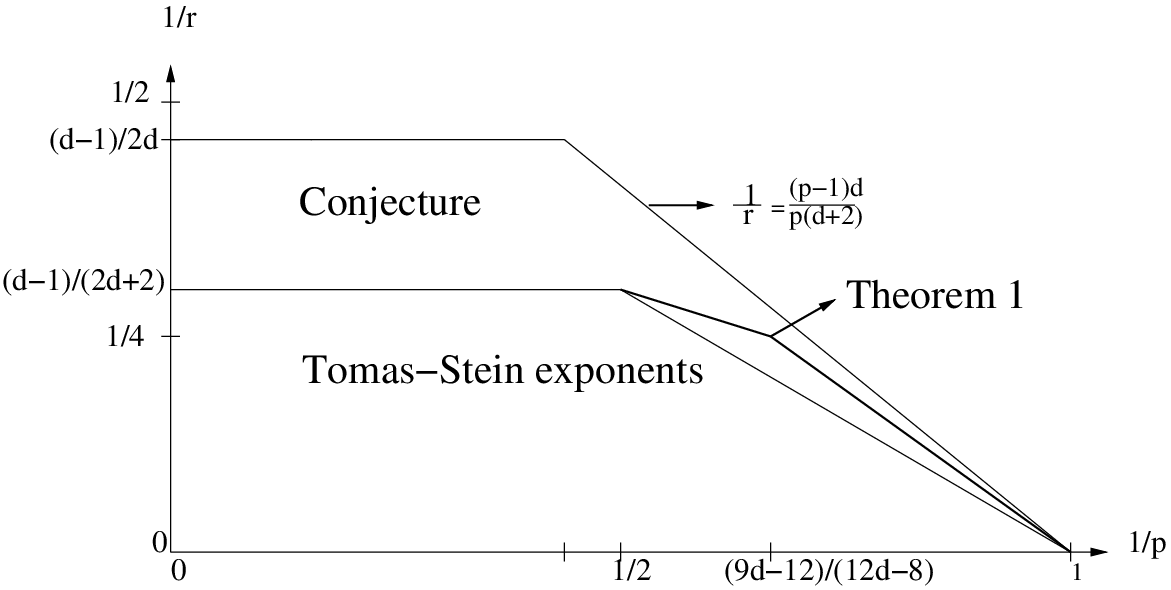}
\caption{\label{Figure3}Conjectured necessary conditions for boundedness of $R^*(p\rightarrow r)$ in even dimensions, $d\geq 4,$ 
and the improved exponents by Theorem \ref{mainThm}. }
\end{figure}

\section{ Results from the classical exponential sums}
In this section, we introduce the consequences driven by using classical bounds for exponential sums, 
such as Gauss sums and generalized Kloosterman sums. 
In the remainder of this paper, we identify sets with their characteristic functions. 
For instance, we write $E$ and $ \widehat{Ed\sigma}$ for  $ \chi_E$ and $ \widehat{\chi_Ed\sigma}$ respectively. 
We also assume that $\chi$ denotes the non-trivial additive character of ${\mathbb F}_q $, 
and we denote by $\eta$ a multiplicative character of ${\mathbb F}_q^* $ of order two,
that is , $\eta(ab) = \eta(a) \eta(b)$ and $\eta^2(a) =1$ for all $ a, b \in {\mathbb F}_q^*$.
Then the estimate for Gauss sums is given by
\[G_a(\chi,\psi)=\sum_{t\in {\mathbb  F}_q^*}\eta(t)\chi(at) \lesssim q^{\frac{1}{2}}, \quad 
\mbox{for} \quad a\in {\mathbb F}_q^*,\] 

the estimate for Kloosterman sums due to Weil (\cite{We48}) is given by
\[ \sum_{t\in {\mathbb F}_q^*}\chi(at+bt^{-1})  \lesssim
q^{\frac{1}{2}} \quad \mbox{for}\quad a,b \in {\mathbb F}_q^*,\]
and the estimate for twisted Kloosterman sums due to Sali\'e (\cite{Sa32}) is given by 
\begin{equation}\label{GKloosterman} \sum_{t\in {\mathbb F}_q^*}\eta(t)\chi(at+bt^{-1}) \lesssim
q^{\frac{1}{2}} \quad \mbox{for} \quad a,b \in {\mathbb F}_q. \end{equation}
Some exponential sums can be expressed in terms of Gauss sums. 
For example, we have the following formula (see \cite{LN97} or \cite{IK04}).
\begin{equation}\label{Gaussformula}\sum_{s\in {\mathbb F}_q}
\chi(ts^2)=\eta(t)G(\eta,\chi) \quad \mbox{for any} \quad t \in {\mathbb F}_q^*,\end{equation}
where $G(\eta,\chi)$ is a Gauss sum given by $G(\eta,\chi)=\sum\limits_{s\in {\mathbb F}_q^*}
\eta(s)\chi(s).$ We now introduce the explicit form of the Fourier transform of spheres in ${\mathbb F}_q^d$,
which can be found in \cite{IR06}, \cite{IK06}, and \cite{HIKR07}. For the reader's convenience we review the formula below.
\begin{lemma}\label{Lem1}
Let $S_j$ be a sphere in ${\mathbb F}_q^d$ defined as in (\ref{defofsphere}).
Then for any $m\in {\mathbb F}_q^d,$ we have
\[ \widehat{S_j}(m) = q^{-1} \delta_0(m) + q^{-d-1}\eta^d(-1) G^d(\eta, \chi) \sum_{r \in {\mathbb F}_q^*} 
\eta^d(r)\chi\Big(jr+ \frac{\|m\|_2}{4r}\Big),\]
where $\delta_0(m)=1$ if $m=(0, \ldots, 0)$ and $\delta_0(m)=0$ otherwise.
\end{lemma} 
\begin{proof} Recall that we write $\widehat{S_j}$ for $ \widehat{\chi_{S_j}}.$ From the definition of
Fourier transform and the orthogonality relations for non-trivial characters, we have
\begin{align*} \widehat{S_j}(m) &= q^{-d} \sum_{x \in {\mathbb F}_q^d} \chi(-x\cdot m) q^{-1} 
\sum_{r\in {\mathbb F}_q} \chi(-r(\|x\|_2 -j))\\
&=q^{-1} \delta_0(m) +q^{-d-1}\sum_{r\in {\mathbb F}_q^*}\chi(jr) 
\sum_{x\in {\mathbb F}_q^d} \chi(-r \|x\|_2 -x\cdot m). \end{align*}
Completing the squares and using the formula in (\ref{Gaussformula}), the proof immediately follows.
\end{proof}
\begin{remark} \label{boundforsphere} From Lemma \ref{Lem1}, we see that $\widehat{S_j}(0,\ldots,0) \approx q^{-1}$ and so the size of 
the sphere $ S_j$ is $\approx q^{d-1}.$ In other words, $|S_j| \approx q^{d-1}.$ Moreover, if $m \ne (0,\ldots,0),$ then 
$|\widehat{S_j}(m)| \lesssim q^{-\frac{d+1}{2}}$ (see \cite{IK07}).
If $d$ is even then $\eta^d =1,$ because $\eta$ is a multiplicative character of ${\mathbb F}_q^*$
of order two. Moreover the exact value of $G^d(\eta, \chi)$ is given by the equation 
\[G^d(\eta, \chi) =K q^{\frac{d}{2}} \quad \mbox{for some} \quad K \in \mathbb C ,\]
where $K$ depends on the additive character $\chi$, the size of ${\mathbb F}_q$, and the dimension of ${\mathbb F}_q^d$.
However it is uniformly bounded by $1$. For the exact value of $K$, see \cite{LN97}. Thus for even $d$ and 
for each $m \in {\mathbb F}_q^d$, we have
\begin{equation}\label{even} 
\widehat{S_j}(m) = q^{-1}\delta_0(m) + Kq^{-\frac{d+2}{2}} \sum_{r\in {\mathbb F}_q^*} \chi\Big(jr + \frac{\|m\|_2}{4r}\Big) ,\end{equation}
where the value of $K$ is uniformly bounded by $1$.\end{remark}
\begin{remark}Throughout the paper, the constant $K $ may change from a line to another line,
but it is uniformly bounded by one.\end{remark}

\section{Overview of the proof of Theorem \ref{mainThm}}\label{section3}

In this section, we shall prove Theorem \ref{mainThm}  by assuming that Theorem \ref{Long} (see Section \ref{sectionlong})  holds.
Using the inequality in (\ref{smallgood}) and the usual dyadic pigeonholing argument, it suffices to show that with $p= \frac{12d-8}{9d-12}$,
\begin{equation}\label{Extension} 
\|\widehat{Ed\sigma}\|_{L^4 ({\mathbb F}_q^d,dm)} \lesssim \|E\|_{L^p (S_j,d\sigma)} , \quad \mbox{for all} \quad E \subset S_j.\end{equation}
Expanding both terms in (\ref{Extension}) and using the fact that $|S_j|\approx q^{d-1}$, it is enough to prove that 
\begin{equation}\label{Easyform}
\Lambda_4(E) \lesssim |E|^{\frac{4}{p}} q^{3d-4} q^{\frac{-4d+4}{p}}\quad \mbox{for}\quad p=\frac{12d-8}{9d-12},
\end{equation}
where $\Lambda_4(E) = \sum\limits_{\substack{(x,y,z,k)\in E^4\\:x+y=z+k}} 1$ (see \cite{IK07}). Note that Theorem \ref{Long} implies that 
if $d\geq 4$ is even and $E$ is any subset of the sphere $S_j$, then 

\[\Lambda_4(E)\lesssim \left\{\begin{array}{ll}   q^{-1} |E|^3 \quad &\mbox {if} \quad q^{\frac{d+2}{2}} \lesssim |E| \lesssim q^{d-1} \\
                       q^{\frac{d-2}{4}}|E|^{\frac{5}{2}} \quad &\mbox{if} \quad q^{\frac{d-1}{2}} \lesssim |E| \lesssim  q^{\frac{d+2}{2}}\\
                      q^{\frac{3d-4}{4}}|E|^{\frac{3}{2}} \quad &\mbox{if} \quad q^{\frac{3d-4}{6}} \lesssim |E| \lesssim  q^{\frac{d-1}{2}}\\
                      |E|^3 \quad &\mbox{if} \quad 1 \lesssim |E| \lesssim q^{\frac{3d-4}{6}}. \end{array}\right. \] 

Using these upper bounds of $\Lambda_4(E)$ depending on the size of the subset $E$ of $S_j$, the inequality in (\ref{Easyform}) follows 
by the direct calculation, and the proof of Theorem \ref{mainThm} is complete if we can prove Theorem \ref{Long}. In Section \ref{sectionlong},
we shall prove Theorem \ref{Long}.

\section{Estimating of dot products}
Given a subset $E$ of a sphere $S_j$, we shall estimate the number of pair $(x,z)\in E^2$ 
such that the dot product $x\cdot z$ is exactly the radius of sphere $S_j$. 
Such an estimate enables us to improve the Tomas-Stein exponents at certain points if the dimension, $ d\geq 4,$ of ${\mathbb F}_q^d$ is even.

\begin{theorem}\label{Best} Let $S_j$ be a sphere in ${\mathbb F}_q^d$ with $d\geq 4$ even, defined as in (\ref{defofsphere}).
If $E$ is any subset of the sphere $S_j, ~j\ne 0,$ then we have 
\[ \sum_{(x,y)\in E^2 : x\cdot y=j} 1 \lesssim q^{-1}|E|^2 + q^{\frac{d-2}{2}}|E| .\]
\end{theorem}
\begin{proof}
We begin by noting that 
\begin{align}\label{Formfix1} 
\sum_{(x,y)\in E^2 : x\cdot y=j} 1 =&\sum_{x,y \in E} \delta_0(x\cdot y-j)\nonumber\\
                                  = &\sum_{x,y \in E} q^{-1} \sum_{s\in {\mathbb F}_q} \chi\left(-s(x\cdot y-j)\right)\nonumber\\
                                  =& q^{-1}|E|^2 +I(j), \end{align}
where \[I(j) = \sum_{x,y \in E} q^{-1} \sum_{s\in {\mathbb F}_q^*} \chi\left(-s(x\cdot y-j)\right).\]
Viewing $I(j)$ as a sum in $x\in E,$ and applying the triangle inequality and the Cauchy-Schwartz inequality, we see that
\begin{align*}
|I(j)|^2 \leq& q^{-2} \left( \sum_{x\in E} |\sum_{y\in E,s\in {\mathbb F}_q^*} \chi\left(-s(x\cdot y-j)\right) | \right)^2.\\
\leq & q^{-2} |E| \sum_{x\in E} | \sum_{y\in E, s\in {\mathbb F}_q^*}\chi\left(-s(x\cdot y-j\right)|^2.
\end{align*}
Since $E \subset S_j$, we see that
\begin{align*}
|I(j)|^2 \leq &q^{-2} |E|\sum_{x\in S_j} \sum_{\substack{y,y'\in E\\ s,s'\in {\mathbb F}_q^*}} \chi\left(-s(x\cdot y -j)\right) 
\chi\left(s'(x\cdot y'-j)\right)\\
   =&q^{d-2} |E| \sum_{\substack{y,y'\in E\\ s,s'\in {\mathbb F}_q^*}} \widehat{S_j}(sy-s'y')\chi\left((s-s')j\right)\\
    =& \sum_{\substack{y,y'\in E\\ s,s'\in {\mathbb F}_q^*\\:y=y',s=s'}} \Bar{G}(y,y',s,s')
      + \sum_{\substack{y,y'\in E\\ s,s'\in {\mathbb F}_q^*\\:y=y',s\neq s'}} \Bar{G}(y,y',s,s')\\
&+\sum_{\substack{y,y'\in E\\ s,s'\in {\mathbb F}_q^*\\:y\neq y'}} \Bar{G}(y,y',s,s') =\Bar{A} +\Bar{B}+\Bar{C},
\end{align*}
where \[\Bar{G}(y,y',s,s') =q^{d-2} |E|\widehat{S_j}(sy-s'y')\chi\left((s-s')j\right).\]
From this and (\ref{Formfix1}), it suffices to show that 
\[\sqrt{\Bar{A} +\Bar{B}+\Bar{C}} \lesssim q^{-1}|E|^2+q^{\frac{d-2}{2}}|E| \qquad \mbox{for all} \qquad
E \subset S_j.\]
In order to complete the proof, we will actually show that $\Bar{A},\Bar{B},$ and $\Bar{C}$ are dominated by
\begin{equation}\label{DN} \approx q^{-2}|E|^4+q^{d-2}|E|^2.\end{equation}
Since $y=y', s=s'$, the term $\Bar{A}$ is easily estimated by 
\[\Bar{A} = \sum_{y\in E , s\in {\mathbb F}_q^*} q^{d-2}|E| \widehat{S_j}(0,\ldots,0) \lesssim q^{d-2}|E|^2,\]
which is dominated by the number in (\ref{DN}) as wanted.
Let us estimate the term $ \Bar{B}$. Since $y=y'$, we have 
\[\Bar{B}= \sum_{\substack{y\in E, s,s'\in {\mathbb F}_q^*\\:s\neq s'}} q^{d-2}|E| \widehat{S_j}\left((s-s')y\right) \chi\left((s-s')j\right).\]
Since $s\neq s'$ and $s,s'\neq 0, (s-s')$ runs through the element of ${\mathbb F}_q^*$ exactly $(q-2)$ times. Thus we see that
\[\Bar{B}= q^{d-2}(q-2)|E| \sum_{y\in E, s\in {\mathbb F}_q^*} \widehat{S_j}(sy) \chi(sj).\]
Note that $sy\neq (0,\ldots,0),$ because $s\neq 0$ and the subset $E$ of sphere $S_j$ doesn't contain the origin.
Therefore, using the formula in (\ref{even}), we have
\[\Bar{B}= Kq^{\frac{d-6}{2}}(q-2)|E| \sum_{y\in E, s,r\in {\mathbb F}_q^*} \chi\left(jr+\frac{js^2}{4r}\right) \chi(sj),\]
where we used the fact that $\| sy\|_2 = s^2j$ for all $y \in S_j$.
After completing the square in $s$-variable, we have 
\begin{align*}
\Bar{B}=& K q^{\frac{d-6}{2}}(q-2) |E| \sum_{y\in E , s,r\in {\mathbb F}_q^*} \chi\left(\frac{j}{4r}(s+2r)^2\right)\\
       =&K q^{\frac{d-6}{2}}(q-2) |E| \sum_{y\in E , s \in {\mathbb F}_q,r\in {\mathbb F}_q^*} \chi\left(\frac{j}{4r}(s+2r)^2\right)\\
        &-K q^{\frac{d-6}{2}}(q-2) |E|\sum_{y\in E,r\in {\mathbb F}_q^*} \chi(jr)\\
=& \Bar{B_1} +\Bar{B_2}.\end{align*}
Recall that we have assumed that $ j$ is not zero, because it is the radius of sphere $S_j.$
Therefore, we see that $ \sum\limits_{r\in {\mathbb F}_q^*} \chi(jr) = -1,$ and so $\Bar{B_2} $ is estimated as
\begin{equation}\label{BarB2} \Bar{B_2}\lesssim q^{\frac{d-4}{2}} |E|^2.\end{equation}
Let us estimate the term $\Bar{B_1}.$ Changing of variables, $s+2r \rightarrow s,$ and using the formula in (\ref{Gaussformula}), we have
\[\Bar{B_1} =K^2 q^{\frac{d-5}{2}}(q-1)|E| \sum_{y\in E, r\in {\mathbb F}_q^*} \eta(\frac{j}{4r}).\]
Since $\sum\limits_{r\in {\mathbb F}_q^*} \eta(\frac{1}{r}) =0,$  $\Bar{B_1}$ is exactly $0$. From this and (\ref{BarB2}), we conclude that
$\Bar{B} \lesssim q^{\frac{d-4}{2}} |E|^2,$ which is dominated by the number in (\ref{DN}) as wanted. Finally, it remains to estimate 
the term $\Bar{C}$. Recall that 
\[\Bar{C} = q^{d-2}|E| \sum_{\substack{y,y'\in E, s,s' \in {\mathbb F}_q^*\\ : y\neq y'}} \widehat{S_j}(sy-s'y')\chi\left((s-s')j\right).\]
Write the term $\Bar{C}$ as two parts
\begin{align*} \Bar{C} = &q^{d-2}|E| \sum_{\substack{y,y'\in E, s,s' \in {\mathbb F}_q^*\\ : y\neq y',sy-s'y'=(0,\ldots,0)}}
\widehat{S_j}(0,\ldots,0) \chi\left((s-s')j\right)\\
&+ q^{d-2}|E| \sum_{\substack{y,y'\in E, s,s' \in {\mathbb F}_q^*\\ : y\neq y',sy-s'y'\neq (0,\ldots,0)}} 
\widehat{S_j}(sy-s'y')\chi\left((s-s')j\right) \\
&= \Bar{C_1} +\Bar{C_2}.
\end{align*}
Let us estimate the term $\Bar{C_1}.$ Note that when $y\neq y'$, $sy-s'y' =(0,\ldots, 0)$ happens only if $s=-s$ and $y=-y'$.
Combining this with the fact that $\widehat{S_j}(0,\ldots, 0) \approx q^{-1}$, we see that $\Bar{C_1} \lesssim q^{d-2}|E|^2,$
which is dominated by the term in (\ref{DN}).
It remains to estimate the term $\Bar{C_2}.$ Since $sy-s'y' \neq (0,\ldots,0)$, using the formula for $\widehat{S_j}$ in (\ref{even}),
we have 
\begin{align*}
\Bar{C_2} = &K q^{\frac{d-6}{2}} |E| \sum_{\substack{y,y'\in E, s,s',r\in {\mathbb F}_q^* \\:y\neq y', sy-s'y' \neq (0,\ldots,0)}}
\chi\left(jr +\frac{\|sy-s'y'\|_2}{4r}\right) \chi\left((s-s')j\right) \\
 =&K q^{\frac{d-6}{2}} |E| \sum_{\substack{y,y'\in E, s,s',r\in {\mathbb F}_q^* \\:y\neq y'}}
\chi\left(jr +\frac{\|sy-s'y'\|_2}{4r}\right) \chi\left((s-s')j\right) \\
   &-K q^{\frac{d-6}{2}} |E| \sum_{\substack{y,y'\in E, s,s',r\in {\mathbb F}_q^* \\:y\neq y', sy-s'y' = (0,\ldots,0)}}
\chi(jr) \chi\left((s-s')j\right) \\
 =& \Bar{C_{21}} +\Bar{C_{22}}.
\end{align*}
To estimate $\Bar{C_{22}}$, note that the sum over $r$-variable is exactly $-1$ and the condition, $y\neq y', sy-s'y'=(0,\ldots, 0)$
happens only if $s=-s$ and $y=-y'$. Thus we have $\Bar{C_{22}}\lesssim q^{\frac{d-4}{2}}|E|^2$ which is dominated by the term in (\ref{DN}). 
To estimate the term $\Bar{C_{21}}$, note that $\|sy-s'y'\|_2 =js^2+js'^2-2ss'y\cdot y'.$ Write
\[ \sum_{\substack{y,y'\in E\\ s,s',r\in {\mathbb F}_q^*\\:y\neq y'}}=
\sum_{\substack{y,y'\in E\\ s,s'\in{\mathbb F}_q ,r\in {\mathbb F}_q^*\\:y\neq y'}}
-\sum_{\substack{y,y'\in E\\ s,r\in {\mathbb F}_q^*,s'=0\\:y\neq y'}}
-\sum_{\substack{y,y'\in E\\ s=0, s',r\in {\mathbb F}_q^*,\\:y\neq y'}}
-\sum_{\substack{y,y'\in E\\ s=0 = s' ,r\in {\mathbb F}_q^*,\\:y\neq y'}}\]
and define $\Bar{\Gamma}(y,y',s,s',r)$ as the value
\[Kq^{\frac{d-6}{2}}|E| \chi\left(jr + \frac{js^2+js'^2-2ss'y\cdot y'}{4r}\right) \chi\left((s-s')j\right).\]
Then we have 
\begin{align*} \Bar{C_{21}} =&\sum_{\substack{y,y'\in E\\ s,s'\in{\mathbb F}_q ,r\in {\mathbb F}_q^*\\:y\neq y'}} \Bar{\Gamma}(y,y',s,s',r)
-\sum_{\substack{y,y'\in E\\ s,r\in {\mathbb F}_q^*,s'=0\\:y\neq y'}}\Bar{\Gamma}(y,y',s,s',r)\\
&-\sum_{\substack{y,y'\in E\\ s=0, s',r\in {\mathbb F}_q^*,\\:y\neq y'}}\Bar{\Gamma}(y,y',s,s',r)
-\sum_{\substack{y,y'\in E\\ s=0 = s' ,r\in {\mathbb F}_q^*,\\:y\neq y'}}\Bar{\Gamma}(y,y',s,s',r)\\
=& I+II + III +IV.\end{align*}
Let us first estimate the term $IV$. Note that $IV$ is estimated as
\begin{align*}IV=-Kq^{\frac{d-6}{2}}|E|\sum_{\substack{y,y'\in E, r\in {\mathbb F}_q^*\\:y\neq y'}} \chi(jr)\\
= K q^{\frac{d-6}{2}}|E|\sum_{y,y'\in E:y\neq y'}1 \lesssim q^{\frac{d-6}{2}}|E|^3,\end{align*} 
which is dominated by the term in (\ref{DN}), because the term in (\ref{DN}) is 
$\gtrsim \sqrt{(q^{-2}|E|^4)(q^{d-2}|E|^2)} \gtrsim q^{\frac{d-6}{2}} |E|^3$ using the fact that
$a+b \geq 2 \sqrt{ab}$  for $ a,b  \geq 0.$
Let us estimate the term $III$. It follows that
\begin{align*}
III =& -Kq^{\frac{d-6}{2}}|E| \sum_{\substack{y,y'\in E, s',r \in {\mathbb F}_q^*\\:y\neq y'}} 
\chi\left(jr+\frac{js'^2}{4r}\right) \chi(-s'j)\\
=&-Kq^{\frac{d-6}{2}}|E| \sum_{\substack{y,y'\in E, s'\in {\mathbb F}_q ,r \in {\mathbb F}_q^*\\:y\neq y'}} 
\chi\left(jr+\frac{js'^2}{4r}\right) \chi(-s'j)\\
&+Kq^{\frac{d-6}{2}}|E| \sum_{\substack{y,y'\in E, r\in {\mathbb F}_q^*\\:y\neq y'}}\chi(jr)\\
 =& III_1+III_2. \end{align*}
Note that the term $III_2$ is dominated by $\approx q^{\frac{d-6}{2}}|E|^3$, because the sum over $r$-variable is exactly $-1$.
To estimate the term $III_1$, completing the square in $s'$-variable and changing of variables, $(s'-2r)\rightarrow s'$, and using 
the formula in (\ref{Gaussformula}), we see that
\[III_1 =-K^2 q^{\frac{d-5}{2}}|E| \sum_{\substack{y,y'\in E, r\in {\mathbb F}_q^*\\:y\neq y'}} \eta(\frac{j}{4r})=0,\]
where the fact, $\sum\limits_{r\in {\mathbb F}_q^*} \eta(\frac{1}{r})=0,$ was used to get the last equality.
From this and the estimate for the term $III_2$, we have $III\lesssim q^{\frac{d-6}{2}}|E|^3$, which is also bounded by the term in (\ref{DN}) as before.
Note that the estimate of the term $II$ is exactly same as that of the term $III$. Thus the term $II$ is also bounded by the term in (\ref{DN}).
To complete the proof, we need to show that the term $I$ is bounded by the term in (\ref{DN}). Recall that the term $I$ is given by the value
\[Kq^{\frac{d-6}{2}}|E| \sum_{\substack{y,y'\in E, s,s'\in {\mathbb F}_q, r\in {\mathbb F}_q^*\\: y\neq y'}}
\chi \left(jr+ \frac{js^2+js'^2-2ss'y\cdot y'}{4r}\right) \chi\left((s-s')j\right).\]
Completing the square in $s$-variable above, changing of variables, $s+(-j^{-1}s'y\cdot y'+2r) \rightarrow s,$ 
and using the formula in (\ref{Gaussformula}),
we see that
\begin{align*}I=&Kq^{\frac{d-5}{2}}|E| \sum_{\substack{y,y'\in E\\ s'\in {\mathbb F}_q, r\in {\mathbb F}_q^*\\: y\neq y'}}
\eta(\frac{j}{4r}) \chi\left(\frac{1}{4r}\left(j-j^{-1}(y\cdot y')^2 \right) s'^2 + (y\cdot y'-j)s'\right) \\
=&Kq^{\frac{d-5}{2}}|E| \sum_{\substack{y,y'\in E\\ s'\in {\mathbb F}_q, r\in {\mathbb F}_q^*\\: y\neq y', y\cdot y'= \pm j}}
\eta(\frac{j}{4r}) \chi\left((y\cdot y'-j)s'\right)\\
+&Kq^{\frac{d-5}{2}}|E|\sum_{\substack{y,y'\in E\\ s'\in {\mathbb F}_q, r\in {\mathbb F}_q^*\\: y\neq y', y\cdot y' \neq \pm j}}
\eta(\frac{j}{4r}) \chi\left(\frac{1}{4r}\left(j-j^{-1}(y\cdot y')^2 \right) s'^2 + (y\cdot y'-j)s'\right)  \\
 =& I_1 +I_2.  \end{align*}
The term $I_1$ is exactly $0$. To see this, observe that  the sum over $r$-variable is zero, because
$\sum\limits_{r\in {\mathbb F}_q^*} \eta(\frac{1}{r}) =0.$ 
To estimate the term $I_2$, completing the square in $s'$-variable, changing of variables, 
$s'+\frac{-2rj}{y\cdot y' +j} \rightarrow s',$ and using the formula in (\ref{Gaussformula}), we see that the value $I_2$ is given by
\[Kq^{\frac{d-4}{2}}|E| \eta(-1)\sum_{\substack {y,y'\in E \\ r\in {\mathbb F}_q^*\\ :y\neq y', y\cdot y' \neq \pm j}}
\eta\left((y\cdot y')^2-j^2\right) \chi\left(\frac{rj(y\cdot y'-j)}{y\cdot y' +j} \right),\]
where we used the fact that $\eta$ is a multiplicative character of ${\mathbb F}_q^*$ of order two. Note that the sum over $r$-variable above
is exactly $-1$. Thus we conclude that $I_2 \lesssim q^{\frac{d-4}{2}}|E|^3,$ which is also dominated by the value in (\ref{DN}), because
the first term in (\ref{DN}) dominates  $q^{\frac{d-4}{2}}|E|^3$  if  $ |E|\gtrsim q^{\frac{d}{2}}$  and the second term in (\ref{DN})
dominates $q^{\frac{d-4}{2}}|E|^3$  if  $ |E| \lesssim q^{\frac{d}{2}}.$
Thus the proof is complete.

\end{proof}

As the direct application of Theorem \ref{Best}, we obtain the following corollary which shall make an important role 
to prove Theorem \ref{Long} in Section 5.
\begin{corollary}\label{Good}
With the same assumptions as in Theorem \ref{Best}, we have
\[ \sum_{\substack{(x,z,z',s, s' ) \in E^3 \times ( {\mathbb F}_q^*)^2 \\: z\neq z' \\ s(-x+z) +s' (x-z')= (0,\ldots,0)}} 1 
\lesssim q|E|^2 + q^\frac{d+2}{2} |E|.\]
\end{corollary}
\begin{proof} Let $E$ be a subset of the sphere $S_j$ defined as in (\ref{defofsphere}).
Suppose $x,z,z' \in E , s,s'\in {\mathbb F}_q^* .$ 
Then we first observe that if $s(-x+z) +s' (x-z')= (0,\ldots,0)$ then $x,z,$and $z'$ must be on a line, because $s,s'\neq 0.$
Moreover if $z\neq z'$ then $s(-x+z) +s' (x-z')= (0,\ldots,0)$ never happens if $x=z$ or $x=z'$.
Thus if $z\neq z' $and $s(-x+z) +s' (x-z')= (0,\ldots,0)$ then $x,z,z'$ are three different  points on a line.
This implies that the line passing through two points $x,z \in E\subset S_j$ should contain one point on $S_j$ which is different from $x$ and $z$.
In other words, it satisfies that $\| x +\alpha (-x+z)\|_2 =j $ for some $\alpha \in {\mathbb F}_q\setminus \{0,1\}.$ 
Since $x,z \in S_j$ , simple calculation of $\| x +\alpha (-x+z)\|_2 =j$ yields that 
$\alpha (\alpha-1)(j-x\cdot z) =0 $ for some $\alpha \in {\mathbb F}_q\setminus \{0,1\}.$
Thus we conclude that if $z\neq z' $and $s(-x+z) +s' (x-z')= (0,\ldots,0)$ then we have $j-x\cdot z=0.$
Using this fact, we see that 
\begin{align*}&\sum_{\substack{(x,z,z',s, s' ) \in E^3 \times ( {\mathbb F}_q^*)^2 \\: z\neq z' \\ s(-x+z) +s' (x-z')= (0,\ldots,0)}} 1 \\
             \leq &\sum_{x,z \in E : x\cdot z=j} \sum_{\substack{z'\in E, s,s'\in{\mathbb F}_q^*\\: s(-x+z) +s' (x-z')= (0,\ldots,0)}} 1\\
              \leq & \sum_{x,z \in E : x\cdot z=j} q^2,\end{align*}
where the last line can be obtained by using the facts that a line has $q$ elements and if we fix $x,z$ then the maximum number of  choices of $z' \in E$ is at most $q,$
because $x,z,z'$ are exactly on  one line , and so if $x,z,z'$ are determined then the number of choices of the pair $(s,s')\in ({\mathbb F}_q^*)^2$ satisfying 
$s(-x+z) +s' (x-z')= (0,\ldots,0)$ is at most $q$. From Theorem \ref{Best}, the proof is complete.

\end{proof}

\section{Incidence theorems}\label{sectionlong}
The purpose of this section is to develop the incidence theory so that we shall obtain the key estimate for the proof of Theorem \ref{mainThm}.
As we observed in Section \ref{section3}, Theorem \ref{Long} below provides the complete proof of Theorem \ref{mainThm}.
\begin{theorem}\label{Long} Let $S_j$ be a sphere in ${\mathbb F}_q^d$ defined as before. In addition, we assume that the dimension of ${\mathbb F}_q^d,d\geq 4,$
is even. If $E$ is any subset of $S_j$ then we have   
\begin{align*}\Lambda_4(E)=&\sum_{\substack{(x,y,z,k) \in E^4 \\: x+y=z+k}} 1 \\
\lesssim & \min\{|E|^3,\quad q^{-1}|E|^3 + q^{\frac{d-2}{4}}|E|^{\frac{5}{2}} 
+ q^{\frac{3d-4}{4}}|E|^\frac{3}{2}\}.\end{align*}
\end{theorem}

\begin{proof} We first note that it is trivial that
$\Lambda_4(E) \lesssim |E|^3,$
because if we fix $x,y,z\in E$ then there is at most one $k$ such that $x+y=z+k.$ Thus it suffices to show that
\[\Lambda_4(E)\lesssim q^{-1}|E|^3 + q^{\frac{d-2}{4}}|E|^{\frac{5}{2}} 
+ q^{\frac{3d-4}{4}}|E|^\frac{3}{2}.\]
Since the set $E$ is a subset of the sphere $S_j$, we see that
\[ \Lambda_4(E)\leq \sum_{\substack{(x,y,z)\in E^3\\: x+y-z\in S_j}} 1 = \sum_{\substack{(x,y,z)\in E^3 \\:\|x+y-z\|_2 =j}} 1.\]
Therefore we need to estimate the number of elements of the following set:
\begin{align*}&\{(x,y,z)\in E^3: \|x+y-z\|_2=j\}\\
=& \{(x,y,z)\in E^3 : x\cdot y-x\cdot z -y\cdot z =-j\},\end{align*}
where we used the fact that $x,y,$ and $z$ are elements of the sphere $S_j.$
It therefore follows that
\begin{align}\label{Lambda} \Lambda_4(E) \leq & \sum_{(x,y,z)\in E^3} \delta_0 (x\cdot y -x\cdot z-y\cdot z+j)\nonumber\\
                             =&\sum_{(x,y,z)\in E^3} q^{-1} \sum_{s\in {\mathbb F}_q } \chi\left(s(x\cdot y -x\cdot z-y\cdot z+j)\right)\nonumber\\
                             =& q^{-1}|E|^3 + R(j), \end{align}
where \[R(j) =\sum_{(x,y,z)\in E^3} q^{-1} \sum_{s\in {\mathbb F}^*_q } \chi\left(s(x\cdot y -x\cdot z-y\cdot z+j)\right).\] 
Thus our work is to find the upper bound of $|R(j)|.$ Viewing $R(j)$ as a sum in $x\in E$, and applying the triangle inequality and 
the Cauchy-Schwartz inequality in $x$-variable, we have
\begin{equation}\label{Error}|R(j)| \leq q^{-1}|E|^{\frac{1}{2}}\left(\sum_{x\in E} M_j(x) \right)^\frac{1}{2},\end{equation}
where
\[ M_j(x) = \left| \sum_{(y,z,s)\in E^2 \times {\mathbb F}_q^*} \chi\left(s(x\cdot y-x\cdot z-y\cdot z +j)\right) \right|^2.\]
Let us estimate $M_j(x) $ for $x\in E.$ Viewing $M_j(x)^\frac{1}{2}$ as a sum in $y\in E$, applying the triangle inequality and 
the Cauchy-Schwartz inequality in $y$-variable, and dominating the sum over $y\in E$ by the sum over $y\in S_j$, we see that $M_j(x)$ is dominated by the value
\begin{align*}&|E|\sum_{y\in S_j}\sum_{\substack{z,z'\in E\\s,s'\in {\mathbb F}_q^*}} \chi \left(s(x\cdot y-x\cdot z-y\cdot z +j)\right)
\chi \left(-s'(x\cdot y -x\cdot z'-y\cdot z'+j)\right)\\
=& q^d|E| \sum_{\substack{z,z'\in E\\s,s'\in {\mathbb F}_q^*}} \widehat{S_j}\left(s(-x+z)+s'(x-z')\right) 
\chi\left(j(s-s')+ s(-x\cdot z)+ s'(x\cdot z')\right).\end{align*}
To estimate $M_j(x)$ for each $x\in E$, we write above sum  into three parts 
\[ \sum_{\substack{z,z'\in E\\ s,s'\in {\mathbb F}_q^*}} = \sum_{\substack{(z,z',s,s')\in E^2\times ({\mathbb F}_q^*)^2\\:z=z',s=s'}}
+\sum_{\substack{(z,z',s,s')\in E^2\times ({\mathbb F}_q^*)^2\\:z=z',s\neq s'}}+\sum_{\substack{(z,z',s,s')\in E^2\times ({\mathbb F}_q^*)^2\\:z\neq z'}}.\]
For $x, z,z'\in E $ and $s,s'\in {\mathbb F}_q^*,$ we denote by $G(x,z,z',s,s')$  the value
\[q^d|E|\widehat{S_j}\left(s(-x+z)+s'(x-z')\right) 
\chi\left(j(s-s')+ s(-x\cdot z)+ s'(x\cdot z')\right).\]
Then we have  the following upper bound of $\sum\limits_{x\in E}M_j(x)$ given as in (\ref{Error}):
\begin{align}\label{DefABC}
\sum_{x\in E}M_j(x) \leq& \sum_{\substack{(x,z,z',s,s')\in E^3\times ({\mathbb F}_q^*)^2\\:z=z',s=s'}} G(x,z,z',s,s')\nonumber\\
                         &+\sum_{\substack{(x,z,z',s,s')\in E^3\times ({\mathbb F}_q^*)^2\\:z=z',s\neq s'}} G(x,z,z',s,s')\nonumber\\
                         &+\sum_{\substack{(x,z,z',s,s')\in E^3\times ({\mathbb F}_q^*)^2\\:z\neq z'}}G(x,z,z',s,s')\nonumber\\
                         =& A_j +B_j +C_j.\end{align}
From (\ref{Lambda}), (\ref{Error}), and (\ref{DefABC}), we see that for each $E \subset S_j$,
\begin{equation}\label{Form1} \Lambda_4(E) \leq q^{-1}|E|^3 +q^{-1}|E|^{\frac{1}{2}} \left( |A_j|^{\frac{1}{2}} + |B_j|^{\frac{1}{2}} + |C_j|^\frac{1}{2} \right).
\end{equation}
In order to complete the proof of Theorem \ref{Long}, we shall carefully estimate the three terms, $A_j,B_j ,C_j .$
Since $\widehat{S_j}(0,\ldots, 0) \approx q^{-1}, z=z',$ and $ s=s',$  $A_j$ is easily estimated as 
\[ A_j =q^d|E| \sum_{(x,z,s)\in E^2\times {\mathbb F}_q^*} \widehat{S_j}(0,\ldots, 0) \lesssim q^d |E|^3.\]
From this and (\ref{Form1}), we obtain that 
\[\Lambda_4(E) \lesssim  q^{-1}|E|^3 +q^{\frac{d-2}{2}} |E|^2 + q^{-1}|E|^{\frac{1}{2}} \left(|B_j|^{\frac{1}{2}} + |C_j|^\frac{1}{2} \right).\]
Using the fact that $ a+b \geq 2 \sqrt{ab} $ if $a,b \geq 0,$ we see that
the term $ q^{\frac{d-2}{2}} |E|^2$ is $\lesssim q^{\frac{d-2}{4}}|E|^{\frac{5}{2}} 
+ q^{\frac{3d-4}{4}}|E|^\frac{3}{2}.$ Therefore, in order to complete the proof, it suffices to show that
$q^{-1}|E|^{\frac{1}{2}} \left(|B_j|^{\frac{1}{2}} + |C_j|^\frac{1}{2} \right) \lesssim q^{\frac{d-2}{4}}|E|^{\frac{5}{2}}
+ q^{\frac{3d-4}{4}}|E|^\frac{3}{2}.$ To show this, we shall prove the following two estimates:
\begin{equation}\label{ConquerBj}
|B_j| \lesssim q^{\frac{d+2}{2}} |E|^4 + q^{\frac{3d}{2}}|E|^2, \end{equation}
and 
\begin{equation}\label{ConquerCj}
|C_j| \lesssim q^{\frac{d+2}{2}} |E|^4 + q^{\frac{3d}{2}}|E|^2. \end{equation}
In the following subsections, we shall estimate the terms, $B_j$ and $C_j $ to prove
above two inequalities.   
\subsection{Estimate of the term $B_j$}
Since $z=z'$ in the sum of the term $B_j$ ,we first see that the term $B_j$ is given by the value
\[q^d|E| \sum_{\substack{(x,z,s,s')\in E^2\times ({\mathbb F}_q^*)^2\\: s\neq s'}} \widehat{S_j}\left((s-s')(-x+z)\right) \chi\left((s-s')(j-x\cdot z)\right).\]
Now observe that $s-s'$ runs through each value in ${\mathbb F}_q^*$ exactly $(q-2)$ times, because $s,s' \neq 0, s\neq s'.$ Therefore we have
\begin{align*} B_j =&q^d(q-2)|E|\sum_{(x,z,s)\in E^2\times {\mathbb F}_q^*} \widehat{S_j}\left(s(-x+z)\right) \chi\left(s(j-x\cdot z)\right)\\
=& q^d(q-2)|E|\sum_{\substack{(x,z,s)\in E^2\times {\mathbb F}_q^*\\
:x=z}} \widehat{S_j}(0,\ldots, 0)\\
 &+ q^d(q-2)|E|\sum_{\substack{(x,z,s)\in E^2\times {\mathbb F}_q^*\\
:x\neq z}} \widehat{S_j}\left(s(-x+z)\right) \chi\left(s(j-x\cdot z)\right) \\
  =&B_{j,1} + B_{j,2}.\end{align*}
 Since $\widehat{S_j}(0,\ldots, 0) \approx q^{-1}$ and $x=z \in E$, we have 
$B_{j,1} \lesssim q^{d+1}|E|^2$ which is dominated by the value in (\ref{ConquerBj}) for $d\geq 2$ as wanted.
To estimate the term $B_{j,2}$, observe that $s(-x+z) \neq (0,\ldots,0),$ because $x\neq z$ and $s\neq 0.$ 
Noting that $\|s(-x+z)\|_2 = (2j-2x\cdot z)s^2$ for $x,z \in E \subset S_j$ 
and using the explicit form for $\widehat{S_j}$ in (\ref{even}), we see that 
\begin{align*} 
B_{j,2} =& K q^{\frac{d-2}{2}}(q-2)|E| \sum_{\substack{(x,z,s,t)\in E^2\times ({\mathbb F}_q^*)^2 \\
: x\neq z}} \chi\left(jt+\frac{(j-x\cdot z)s^2}{2t}\right) \chi\left((j-x\cdot z) s\right) \\
=& K q^{\frac{d-2}{2}}(q-2)|E| \sum_{\substack{(x,z,s,t)\in E^2 \times ({\mathbb F}_q^*)^2 \\: x\neq z, j-x\cdot z=0}} \chi(jt) \\
+& Kq^{\frac{d-2}{2}}(q-2)|E| \sum_{\substack{(x,z,s,t)\in E^2\times ({\mathbb F}_q^*)^2 \\: x\neq z, j-x\cdot z\neq 0}} 
\chi\left(jt+\frac{(j-x\cdot z)s^2}{2t}\right) \chi\left((j-x\cdot z) s\right) \\
 &= B_{j,21} + B_{j,22},\end{align*}
where $K$ is the complex value which is bounded by $1$. Recall that ``$j$" is the radius of the sphere $S_j$, and not zero. 
Therefore, $\sum\limits_{t\in {\mathbb F}_q^*} \chi(jt)$ 
is exactly $-1$ and so the term $B_{j,21}$ is written by the value
\[ B_{j,21} = -Kq^{\frac{d-2}{2}}(q-1)(q-2)|E| \sum_{\substack{(x,z) \in E^2\\: x\neq z, j-x\cdot z =0}} 1.\]
Using Theorem \ref{Best}, we obtain that 
\[B_{j,21} \lesssim q^{\frac{d}{2}}|E|^3+q^d|E|^2.\]
Thus $B_{j,21}$ is also dominated by the value in (\ref{ConquerBj}) as wanted.
To estimate the term $B_{j,22}$, write the term $B_{j,22}$ as follows:
\begin{align*}
 &Kq^{\frac{d-2}{2}}(q-2)|E| \sum_{\substack{(x,z,s,t)\in E^2\times {\mathbb F}_q\times {\mathbb F}_q^* \\: x\neq z, j-x\cdot z\neq 0}} 
\chi\left(jt+\frac{(j-x\cdot z)s^2}{2t}\right) \chi\left((j-x\cdot z) s\right) \\
&-Kq^{\frac{d-2}{2}}(q-2)|E|\sum_{\substack{(x,z,t)\in E^2\times {\mathbb F}_q^* \\: x\neq z, j-x\cdot z\neq 0}} \chi(jt) =B_{j,221}+B_{j,222}.\end{align*}
Then the term $B_{j,222}$ is clearly dominated by $\approx q^{\frac{d}{2}} |E|^3,$ 
because the sum over $t\in {\mathbb F}_q^* $ above is exactly -1. Thus the term  $B_{j,222}$ is clearly dominated by the value in (\ref{ConquerBj}).
To justify  the inequality in (\ref{ConquerBj}), it remains to show that the term $B_{j,221}$ is dominated by the value in (\ref{ConquerBj}).
Let us estimate the term $B_{j,221}.$ After completing the square in $s$-variable and changing of variables, $s+t\rightarrow s,$ we see that
\[B_{j,221}=kq^{\frac{d-2}{2}}(q-2)|E| \sum_{\substack{(x,z,s,t)\in E^2\times {\mathbb F}_q\times {\mathbb F}_q^* \\: x\neq z, j-x\cdot z\neq 0}} 
\chi\left(\frac{(j-x\cdot z)s^2}{2t}\right)\chi\left(\frac{(j+x\cdot z)t}{2}\right),\]
Since $j-x\cdot z\neq 0$ and $t\neq 0$, we can apply the formula in (\ref{Gaussformula}) to get the Gauss sum from the sum over $s$-variable.
As a consequence, we obtain that 
\[B_{j,221}=k^2q^{\frac{d-1}{2}}(q-2)|E| \sum_{\substack{(x,z,t)\in E^2\times {\mathbb F}_q^* \\: x\neq z, j-x\cdot z\neq 0}} 
\eta\left(\frac{j-x\cdot z}{2t}\right) \chi\left(\frac{(j+x\cdot z)t}{2}\right).\]
where $\eta$ is the multiplicative character of $ {\mathbb F}_q^*$ of order two.
Note that $\eta(\frac{1}{t}) =\eta(t)$ for all $t\in {\mathbb F}_q^*$, because the order of the character $\eta$ is two.
Then we see that the sum over $t$-variable above is just one of the twisted Kloosterman sums introduced in (\ref{GKloosterman}).
Thus we obtain that
\[\label{Bj221} B_{j,221} \lesssim q^{\frac{d+2}{2}}|E|^3.\]
This clearly implies that $B_{j,221} $ is dominated by the value in (\ref{ConquerBj}) and so the inequality in (\ref{ConquerBj}) holds.

\subsection{Estimate of the term $C_j$}
In order to complete the proof of Theorem\ref{Long}, we shall estimate the term $C_j$ given by
\[\sum_{\substack{(x,z,z',s,s')\in E^3\times ({\mathbb F}_q^*)^2 \\: z\neq z'}} G(x,z,z',s,s') ,\]
where $G(x,z,z',s,s')$ is defined by the value 
\[q^d|E| \widehat{S_j}\left(s(-x+z)+s'(x-z')\right) 
\chi\left(j(s-s')+ s(-x\cdot z)+ s'(x\cdot z')\right).\]
Recall that we need to prove that the inequality in (\ref{ConquerCj}) holds to complete the proof of Theorem \ref{Long}.
Let us begin by writing the term $C_j$ as two parts 
\begin{align*}
 &\sum_{\substack{(x,z,z',s,s')\in E^3\times ({\mathbb F}_q^*)^2 \\: z\neq z',\\ s(-x+z)+s'(x-z')= (0,\ldots,0)}}
q^d|E|\widehat{S_j}(0,\ldots,0) 
\chi\left(j(s-s')+ s(-x\cdot z)+ s'(x\cdot z')\right)\\
+ &\sum_{\substack{(x,z,z',s,s')\in E^3\times ({\mathbb F}_q^*)^2 \\: z\neq z'\\s(-x+z)+s'(x-z')\neq (0,\ldots,0)}}
G(x,z,z',s,s') =C_{j,1}+C_{j,2}.\end{align*}
We shall first estimate the term $C_{j,1}$. The condition $s(-x+z)+s'(x-z')= (0,\ldots,0)$  clearly implies that
$x\cdot \left(s(-x+z)+s'(x-z')\right)= (0,\ldots,0).$ By the direct calculation of the dot product, we see that 
$j(s-s')+ s(-x\cdot z)+ s'(x\cdot z')=0,$ because $x\in E \subset S_j.$ Using this and  the fact that $\widehat{S_j}(0,\ldots,0) \approx q^{-1},$
we obtain that 
\[C_{j,1} \approx q^{d-1}|E| \sum_{\substack{(x,z,z',s,s')\in E^3\times ({\mathbb F}_q^*)^2 \\: z\neq z',\\ s(-x+z)+s'(x-z')= (0,\ldots,0)}} 1.\]
Combining this with Corollary \ref{Good} , we conclude that
\[C_{j,1} \lesssim q^d |E|^3 + q^{\frac{3d}{2}} |E|^2.\]
Note that the term $ C_{j,1} $ is dominated by the value in (\ref{ConquerCj}). To see this, use the fact that $a+b \geq 2 \sqrt{ab}$
for $a,b \geq 0$ so that the term $ q^d |E|^3 $ is dominated by $\approx q^{\frac{d+2}{2}} |E|^4 + q^{\frac{3d}{2}}|E|^2.$
We shall estimate the term $C_{j,2}.$ Since $s(-x+z)+s'(x-z')\neq (0,\ldots,0)$, using the explicit form for $\widehat{S_j}$ in (\ref{even}) yields
that \[C_{j,2} =K q^{\frac{d-2}{2}}|E|\sum_{\substack{(x,z,z',s,s',t)\in E^3\times ({\mathbb F}_q^*)^3 \\: z\neq z'\\s(-x+z)+s'(x-z')\neq (0,\ldots,0)}}
 \Omega_j(x,z,z',s,s',t),\] where $\Omega_j(x,z,z',s,s',t)$ is given by
\[\chi \left( jt+ \frac{\|s(-x+z)+s'(x-z')\|_2 }{4t}\right)
\chi\left( j(s-s')+s(-x\cdot z) +s'(x\cdot z')\right).\]
To eliminate the condition $s(-x+z)+s'(x-z')\neq (0,\ldots,0)$, we rewrite the term $C_{j,2}$ as following:
\begin{align*}
C_{j,2}=& K q^{\frac{d-2}{2}}|E|\sum_{\substack{(x,z,z',s,s',t)\in E^3\times ({\mathbb F}_q^*)^3 \\: z\neq z'}}
 \Omega_j(x,z,z',s,s',t)\\
&- K q^{\frac{d-2}{2}}|E|\sum_{\substack{(x,z,z',s,s',t)\in E^3\times ({\mathbb F}_q^*)^3 \\: z\neq z'\\s(-x+z)+s'(x-z')= (0,\ldots,0)}}
 \Omega_j(x,z,z',s,s',t)\\
    &= C_{j,21} + C_{j,22}.\end{align*}
Using the arguments for the estimate of $C_{j,1}$ as before, we can easily obtain 
\[ C_{j,22}\lesssim q^{\frac{d}{2}}|E|^3+  q^d|E|^2,\]
where we also used the fact that $ \sum\limits_{t \in {\mathbb F}_q^* } \chi(jt) = -1.$
Thus the term $C_{j,22}$ is clearly dominated by the value in (\ref{ConquerCj}).
Let us estimate the term $ C_{j,21}$. For the simple notation, let $P=j-x\cdot z, Q=j-x\cdot z',$ and $ U=z\cdot z'-j.$ 
Then direct calculation shows that \[j(s-s')+s(-x\cdot z) + s'(x\cdot z')= Ps-Qs',\] 
and 
\[\frac{\|s(-x+z)+s'(x-z')\|_2}{4t} = \frac{Ps^2 - (P+Q+U)ss'+Qs'^2}{2t}.\]
Thus the term $C_{j,21}$ is given by the value
\[K q^{\frac{d-2}{2}}|E| \sum_{\substack{ x,z,z' \in E \\ s,s',t\in {\mathbb F}_q^* \\ : z\neq z'}}
\chi(jt) \chi \left(\frac{Ps^2-(P+Q+U)ss'+2Pts+Qs'^2-2Qts'}{2t}\right).\]
Define $\Gamma_j(x,z,z',s,s',t)$ as the following value:
\[K q^{\frac{d-2}{2}}|E|\chi(jt) \chi \left(\frac{Ps^2-(P+Q+U)ss'+2Pts+Qs'^2-2Qts'}{2t}\right),\]
and write $C_{j,21}$ as four terms:
\begin{align*} C_{j,21}=& \sum_{\substack{ x,z,z' \in E \\ s,s',t\in {\mathbb F}_q^* \\ : z\neq z'\\ P=Q=0}}\Gamma_j(x,z,z',s,s',t) +
\sum_{\substack{ x,z,z' \in E \\ s,s',t\in {\mathbb F}_q^* \\ : z\neq z'\\P\neq 0=Q}}\Gamma_j(x,z,z',s,s',t)\\
+&\sum_{\substack{ x,z,z' \in E \\ s,s',t\in {\mathbb F}_q^* \\ : z\neq z'\\P=0\neq Q}}\Gamma_j(x,z,z',s,s',t)
+\sum_{\substack{ x,z,z' \in E \\ s,s',t\in {\mathbb F}_q^* \\ : z\neq z'\\ P\neq 0\neq Q}}\Gamma_j(x,z,z',s,s',t)\\
=& C_{j,211} +C_{j,212} +C_{j,213} +C_{j,214}.
\end{align*}
We shall estimate the term $C_{j,211}.$ From the condition $P=Q=0$, we see that the value $C_{j,211}$ is given by
\[C_{j,211}= Kq^{\frac{d-2}{2}}|E| \sum_{\substack{ x,z,z' \in E \\ s,s',t\in {\mathbb F}_q^* \\ : z\neq z'\\ P=Q=0}}\chi(jt)  
\chi\left(\frac{-Uss'}{2t}\right).\]
We remark that the values $P,Q,$ and $U$ are independent of the variables $s,s',t \in {\mathbb F}_q^*.$
In case $U=0$, we claim that the contribution to the bound of the term $C_{j,211}$ is given by 
\begin{equation}\label{EU0}q^{\frac{d}{2}}|E|^4+ q^d|E|^3.\end{equation}
To justify the claim, note that the sum over $t\in {\mathbb F}_q^*$ is exactly $-1$ in case $U=0.$ It therefore follows that
\begin{align*} Kq^{\frac{d-2}{2}}|E|\sum_{\substack{ x,z,z' \in E , s,s',t\in {\mathbb F}_q^* \\ : z\neq z', P=Q=U=0}} \chi(jt)
\lesssim&  q^{\frac{d-2}{2}}|E|\sum_{\substack{ x,z,z' \in E , s,s'\in {\mathbb F}_q^* \\ : z\neq z', P=Q=U=0}} 1\\
\lesssim &q^{\frac{d+2}{2}}|E|^2 \sum_{\substack{ x,z \in E \\:x\cdot z =j}} 1\\
\lesssim & q^{\frac{d}{2}}|E|^4+ q^d|E|^3,\end{align*}
where we used the fact that $P=j-x\cdot z=0$ in the second inequality, and Theorem \ref{Best} in the last inequality.
Thus the claim in (\ref{EU0}) is complete.
On the other hand, if $U\neq 0$, the contribution to the bound of the term $C_{j,211}$ is given by the value
\begin{equation}\label{EUN0} q^{\frac{d-2}{2}}|E|^4+ q^{d-1}|E|^3.\end{equation}
To see this, note that if $U\neq 0$ then after changing of variables, $\frac{-Uss'}{2t} \rightarrow s'$, we see that
\begin{align*}
&Kq^{\frac{d-2}{2}}|E| \sum_{\substack{ x,z,z' \in E , s,s',t\in {\mathbb F}_q^* \\ : z\neq z', P=Q=0, U\neq 0}}\chi(jt)  
\chi\left(\frac{-Uss'}{2t}\right)\\ 
\approx& q^{\frac{d}{2}}|E|\sum_{\substack{ x,z,z' \in E \\ : z\neq z', P=Q=0, U\neq 0}}1
\lesssim q^{\frac{d}{2}}|E|^2 \sum_{\substack{ x,z \in E \\:x\cdot z =j}}1\\
\lesssim& q^{\frac{d-2}{2}}|E|^4+ q^{d-1}|E|^3,\end{align*}
where we also used Theorem \ref{Best}.
Combining (\ref{EU0}) with (\ref{EUN0}), we obtain that 
\[ C_{j,211} \lesssim q^{\frac{d}{2}}|E|^4+ q^d|E|^3.\]
Thus the term $C_{j,211} $ is also dominated by the term in (\ref{ConquerCj}), because the term $q^d|E|^3$ is dominated by 
$\approx q^{\frac{d+2}{2}}|E|^4 + q^{\frac{3d}{2}}|E|^2$ using the fact that $a+b\geq 2\sqrt{ab}$ if $a,b\geq0$ as before.
Let us estimate the term $C_{j,212}.$  Since $P\neq 0 $ and $Q= 0$, the term $C_{j,212}$ is given by
\[Kq^{\frac{d-2}{2}}|E|\sum_{\substack{ x,z,z' \in E , s,s',t\in {\mathbb F}_q^* \\ : z\neq z', P\neq 0, Q= 0}}
\chi(jt) \chi \left(\frac{Ps^2-(P+U)ss'+2Pts}{2t}\right),\]
where we recall that $P=j-x\cdot z , Q=j-x\cdot z'$ and $ U=z\cdot z'-j.$
We now write the term $C_{j,212}$ as following:
\begin{align*}
&Kq^{\frac{d-2}{2}}|E|\sum_{\substack{ x,z,z' \in E , s\in {\mathbb F}_q ,s',t\in {\mathbb F}_q^* \\ : z\neq z', P\neq 0, Q= 0}}
\chi(jt) \chi \left(\frac{Ps^2-(P+U)ss'+2Pts}{2t}\right)\\
-&Kq^{\frac{d-2}{2}}|E|\sum_{\substack{ x,z,z' \in E  ,s',t\in {\mathbb F}_q^* \\ : z\neq z', P\neq 0, Q= 0}}\chi(jt) =I +II
\end{align*}
As before, the term $II$ is easily estimated as 
\begin{align*}
II\lesssim & \quad q^{\frac{d}{2}}|E|^2 \sum_{ x,z' \in E : x\cdot z' =j} 1\\
\lesssim & \quad q^{\frac{d-2}{2}}|E|^4+ q^{d-1}|E|^3,\end{align*}
where Theorem \ref{Best} was used in the last line. As before, we therefore see that the term $II$ is also dominated by
the term in (\ref{ConquerCj}).
We shall estimate the term $I.$ Rewrite the term $I$ as following:
\begin{align*} I =&Kq^{\frac{d-2}{2}}|E|
\sum_{\substack{ x,z,z' \in E, s\in {\mathbb F}_q, s',t\in {\mathbb F}_q^* \\ : z\neq z', P\neq 0, Q= 0=P+U}}\chi(jt) \chi\left(\frac{Ps^2 +2Pts}{2t}\right)\\
&+Kq^{\frac{d-2}{2}}|E|
\sum_{\substack{ x,z,z' \in E, s\in {\mathbb F}_q, s',t\in {\mathbb F}_q^* \\ : z\neq z', P\neq 0, Q=0,P+U\neq 0}}
\chi(jt) \chi\left(\frac{Ps^2 -(P+U)ss'+2Pts}{2t}\right)\\
=& I_1 +I_2.\end{align*}
Let us estimate the term $I_1$. Completing the square in $s$-variable and changing of variables, $s+t\rightarrow s$, we see that
\[I_1=Kq^{\frac{d-2}{2}}|E|
\sum_{\substack{ x,z,z' \in E, s\in {\mathbb F}_q, s',t\in {\mathbb F}_q^* \\ : z\neq z', P\neq 0, Q=0= P+U}}
\chi(jt) \chi(\frac{Ps^2}{2t}) \chi(\frac{-Pt}{2}).\]
Since $P\neq 0$ and $t\neq 0$, we can apply the formula in (\ref{Gaussformula}) to obtain the Gauss sum from the exponential sum in $s$-variable.
As a result, we have 
\[I_1 = K^2q^{\frac{d-1}{2}}|E|
\sum_{\substack{ x,z,z' \in E, s',t\in {\mathbb F}_q^* \\ : z\neq z', P\neq 0, Q=0= P+U}}
\eta(\frac{P}{2t}) \chi(jt+\frac{-Pt}{2} ),\]
where we used the fact that the Gauss sum is exactly given by $K q^{\frac{1}{2}}$ for some $K\in \mathbb C.$
Notice that the sum over $t$-variable above is a twisted Kloosterman sum introduced as in (\ref{GKloosterman}).
From this and Theorem \ref{Best}, $I_1$ is easily estimated as
\begin{align*} \label{Koh} I_1 \lesssim& q^{\frac{d+2}{2}}|E|\sum_{\substack{ x,z,z' \in E\\ j-x\cdot z'=0}}1\\
                   \lesssim& q^{\frac{d}{2}}|E|^4+ q^{d}|E|^3,\end{align*} 
which is clearly dominated by the value in (\ref{ConquerCj}).
We shall estimate the term $I_2.$ 
Since $P\neq 0, P+U\neq 0, s\neq 0\neq t$, after changing of variables, $\frac{-(P+U)ss'}{2t}\rightarrow s'$, and changing of variables,
$s+t\rightarrow s,$ we see that 
\[ I_2 =Kq^{\frac{d-2}{2}}|E|
\sum_{\substack{ x,z,z' \in E, s\in {\mathbb F}_q, s',t\in {\mathbb F}_q^* \\ : z\neq z', P\neq 0, Q=0,P+U\neq 0}}
\chi(\frac{Ps^2}{2t}) \chi(s')\chi\left((j-\frac{P}{2})t\right).\]
Observe that the sum over $s' \in {\mathbb F}_q^* $ is exactly $-1,$  and use the formula in (\ref{Gaussformula}) to obtain the Gauss sum.
Then we see that
\[ I_2 =-K^2q^{\frac{d-1}{2}}|E|
\sum_{\substack{ x,z,z' \in E, t\in {\mathbb F}_q^* \\ : z\neq z', P\neq 0, Q=0,P+U\neq 0}}
\eta(\frac{P}{2t})\chi\left((j-\frac{P}{2})t\right).\]
Observe that $ \eta(\frac{1}{t}) =\eta(t)$ for $t\in {\mathbb F}_q^*$, because $\eta$ is multiplicative character of ${\mathbb F}_q^*$ of order two.
Then the sum over $t$-variable above is just a twisted Kloosterman sum and so we obtain that 
\begin{align*}I_2 \lesssim &q^{\frac{d}{2}}|E|\sum_{\substack{ x,z,z' \in E\\:x\cdot z'=j}} 1\\
                  \lesssim& q^{\frac{d-2}{2}}|E|^4+ q^{d-1}|E|^3,\end{align*}
where we used Theorem \ref{Best} in the last inequality. Note that the term $I_2$ is also dominated by the value in (\ref{ConquerCj}).
It remains to estimate the terms $C_{j,213}$ and $C_{j,214}.$To get the upper bound of $ C_{j,213}$, 
modify the processes used to obtain the upper bound of $ C_{j,212}$ after switching the role of $P$ and $Q$.
Then we can show that the term $C_{j,213}$ has the same upper bound as  $C_{j,212}$. Thus we also see that the term $C_{j,213}$ is dominated by
the term in (\ref{ConquerCj}).
Finally, we shall estimate the term $C_{j,214}$. Recall that the term $C_{j,214}$ is given by the value
\[K q^{\frac{d-2}{2}}|E|\sum_{\substack{ x,z,z' \in E \\ s,s',t\in {\mathbb F}_q^* \\ : z\neq z'\\ P\neq 0\neq Q}}
\chi(jt) \chi \left(\frac{Ps^2-(P+Q+U)ss'+2Pts+Qs'^2-2Qts'}{2t}\right).\]
We now write the term $C_{j,214}$ by the two terms
\begin{align*}
&K q^{\frac{d-2}{2}}|E|\sum_{\substack{ x,z,z' \in E \\ s\in{\mathbb F}_q \\s',t\in {\mathbb F}_q^* \\ : z\neq z'\\ P\neq 0\neq Q}}
\chi(jt) \chi \left(\frac{Ps^2-(P+Q+U)ss'+2Pts+Qs'^2-2Qts'}{2t}\right)\\
-&K q^{\frac{d-2}{2}}|E|\sum_{\substack{ x,z,z' \in E \\ s',t\in {\mathbb F}_q^* \\ : z\neq z'\\ P\neq 0\neq Q}}
\chi(jt) \chi \left(\frac{Qs'^2-2Qts'}{2t}\right) =G+H.\end{align*}
Let us estimate the term $G$. Letting $T = P+Q+U$ and completing the square over $s$-variable, we note that
\begin{align*} &Ps^2 -(P+Q+U)ss'+2Pts+Qs'^2-2Qts'\\
=&P\left(s+\frac{-Ts'+2tP}{2P} \right)^2 + \frac{(4PQ-T^2)s'^2 +(-8PQt+4PTt)s'-4P^2t^2}{4P}.\end{align*}
After changing of variables, $s+\frac{-Ts'+2tP}{2P}\rightarrow s,$ and using the formula in (\ref{Gaussformula}), we have $G =K^2q^{\frac{d-1}{2}}|E|$
\[\times\sum_{\substack{ x,z,z' \in E ,s',t\in {\mathbb F}_q^* \\ : z\neq z'\\ P\neq 0\neq Q}} \eta(\frac{P}{2t})\chi\left(\frac{(-2Q+T)s'}{2}\right)
\chi\left((j-\frac{P}{2})t + \frac{(4PQ-T^2)s'^2}{8Pt}\right).\]
Observe that the sum over $t$-variable is a twisted Kloosterman sum which is bounded by $\approx q^\frac{1}{2}.$ Thus, we obtain that
\begin{equation}\label{EG} G\lesssim q^{\frac{d+2}{2}}|E|^4.\end{equation}
To get the upper bound of the term $H$, we use the trivial estimate so that we can easily obtain that
\[ H \lesssim q^{\frac{d+2}{2}}|E|^4.\]
From this estimate and (\ref{EG}), we conclude that 
 \[C_{j,214} \lesssim q^{\frac{d+2}{2}}|E|^4.\]
Thus the term $C_{j,214} $ is dominated by the value in (\ref{ConquerCj}) and the proof is complete.

\end{proof}

  \end{document}